\documentclass[12pt]{article}
\usepackage{amsfonts,amssymb,amsbsy,amsmath,amsthm,mathrsfs,enumerate,verbatim}
\usepackage[mathscr]{eucal}

 \usepackage[colorlinks=true]{hyperref}
 \usepackage{graphicx}

 \hypersetup{urlcolor=blue, citecolor=red}
\newtheorem{theo}{Theorem}[section]
\newtheorem{prop}[theo]{Proposition}

\newtheorem{defi}[theo]{Definition}
\theoremstyle{definition}
\newtheorem{rema}[theo]{Remark}

\newcommand{\bel}{\begin{equation} \label}
\newcommand{\ee}{\end{equation}}

\newcommand{\R}{{\mathbb R}}

\def\beq{\begin{equation}}
\def\eeq{\end{equation}}
\newcommand{\bea}{\begin{eqnarray}}
\newcommand{\eea}{\end{eqnarray}}
\newcommand{\beas}{\begin{eqnarray*}}
\newcommand{\eeas}{\end{eqnarray*}}

{

\begin{document}

\begin{center}
{\Large \bf Spectral monodromy of small non-selfadjoint quantum perturbations of completely integrable Hamiltonians}
\medskip

(\today, provisional version)

\end{center}

\medskip

\begin{center}
{\footnote{FITA, Vietnam National University of Agriculture, Hanoi, Vietnam \& Vietnam Institute for Advanced Study in Mathematics, Hanoi, Vietnam //
E-mail: pqsang@vnua.edu.vn}{Quang Sang PHAN}}
\end{center}

\begin{abstract} ~\\
We define a monodromy, directly from the spectrum of small non-selfadjoint perturbations of a selfadjoint semiclassical operator with two degrees of freedom, which is classically integrable. It is a combinatorial invariant that obstructs globally the existence of lattice structure of the spectrum, in the semiclassical limit. Moreover this spectral monodromy allows to recover a topological invariant (the classical monodromy) of the corresponding integrable system.
\end{abstract}

\medskip

{\bf  AMS 2010 Mathematics Subject Classification:} 35P20, 81Q12.

{\bf  Keywords:} integrable system, non-selfadjoint, spectral asymptotic, pseudo-differential operators.

\tableofcontents


\section{Introduction}



\subsection{Motivation}

We propose in this article a way of detecting the monodromy of a quantum Hamiltonian which is classically integrable, by looking at non-selfadjoint perturbations.

In the classical theory, the classical monodromy is defined for a completely integrable system on symplectic manifolds as a topological invariant that obstructs the existence of global action-angle coordinates on the phase space, see \cite{Duis80}.

Quantum monodromy was detected a long time ago in \cite{Cush88} and completely defined in \cite{Vu-Ngoc99}, in the joint spectrum of system of selfadjoint operators that commute, in the sense of the semiclassical limit, as the classical monodromy of the underlying classical system.

However, a mysterious question is whether a monodromy can be defined for only one semiclassical operator? That is how to detect the modification of action-angle variables from only one spectrum?

We are interested in the globally structure of the spectrum of non-selfadjoint $h-$Weyl-pseudodifferential operators with two degrees of freedom, which are small non-selfadjoint perturbations of a selfadjoint operator, in the semiclassical limit.
Such an operator is of the form
 \begin{equation} \label{tt}
  P_{\varepsilon}= P(x,hD_x,\varepsilon; h ), \end{equation}
where the unperturbed operator $P:= P_{\varepsilon=0}$ is formally selfadjoint, and $\varepsilon$ is a small parameter.
In this work $\varepsilon$ is assumed to depend on the classical parameter $h$ and in the regime $h \ll \varepsilon = \mathcal{O}(h^\delta)$, with $0< \delta <1 $.

The first answer for the above problem, which was given in Ref. \cite{QS14}, is a particular case of the present work. In that work, the operators had the simple form $P+ i  \varepsilon Q$, such that the corresponding principal symbols $p$ of $P$ and $q $ of $Q$ commute for the Poisson bracket, $\{p,q\}=0$.
Here we develop this result in assuming only that the principal symbol of $P_{\varepsilon}$ in \eqref{tt} is of the form
$$p_{\varepsilon}=p+ i  \varepsilon q +\mathcal O(\varepsilon^2),$$ with $p$ is a completely integrable Hamiltonian.

It is known from the spectral asymptotic theory (see \cite{Hitrik07}) that, under some suitable global assumption, the spectrum of the perturbed operators has locally the form of a deformed discrete lattice. The eigenvalues admit asymptotic expansions in $h$ and $\varepsilon$.
Moreover, we shaw prove the mail result (Theorem \ref{mtheo1}) that the spectrum is an \emph{asymptotic pseudo-lattice} (see Definition \ref{pseu-lattice}). Therefore, as an application
 from \cite{QS14}, a combinatorial invariant of the spectral lattice- the \emph{spectral monodromy} is well defined, directly from the spectrum.

Moreover, this quantum result is strictly related to the classical results. The spectral monodromy can be identified to the classical monodromy of the completely integrable system $p$.

\subsection{Brief description for the spectral monodromy }



It knows from the spectral asymptotic theory (see \cite{Hitrik07}) that, under an ellipticity condition at infinity (see \eqref{el con}), the spectrum of the perturbed operators \eqref{tt} is discrete, and included in a horizontal band of size $\mathcal{O}(\varepsilon)$.
Moreover, that work allows give asymptotic expansions of the spectrum located in some small domains of size $\mathcal{O} (h ^\delta) \times \mathcal{O }(\varepsilon h^\delta)$ of the spectral band, called \emph{good rectangles}.

As we shaw show in this work that there is a correspondence (in fact a local diffeomorphism, see Proposition \ref{D.A inverse} and a proof for this point in \cite{QS14}), denoted by $f$, from the spectrum contained in a good rectangle $R^{(a)}(\varepsilon,h)$ to a part of $h \mathbb Z^2 $, modulo $\mathcal O(h^\infty)$,
 \begin{eqnarray}
    R^{(a)}(\varepsilon,h) \ni  \mu  \mapsto    f(\mu,\varepsilon; h) \in h \mathbb Z^2 +\mathcal O(h^\infty).
\end{eqnarray}
Here $a \in \mathbb R^2$ is introduced to fix the center of the good rectangle. This map is called a \emph{$h-$local chart} of the spectrum.
The spectrum therefore has the structure of a deformed lattice, with horizontal spacing $h$ and vertical spacing $\varepsilon h$.
See Figure 1 below.
\begin{figure}[!h]
\begin{center}
\includegraphics[width=0.8 \textwidth]{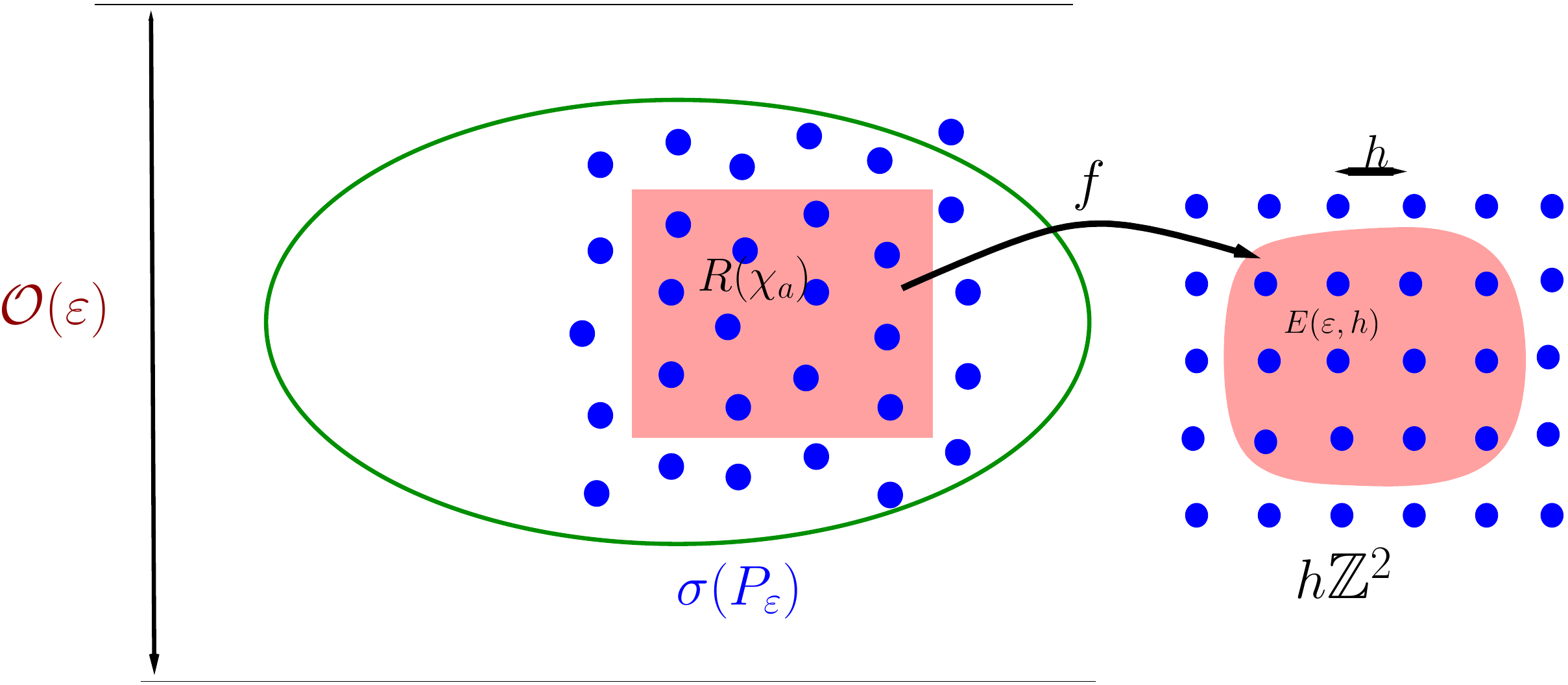}

Figure 1. $h-$local chart of the spectrum
\end{center}
\end{figure}

Such good rectangles correspond to the Diophantine invariant tori in the phase space, on which the Hamiltonian flow of the unperturbed part, that is $p$, is quasi-periodic of constant frequency, see \eqref{frequence}.
However, there are many such rectangles in the spectral band due to the existence of Diophantine invariant tori in the phase space. The set of the center of the good rectangles is nowhere dense in the complex plane, but of large measure. It is well known from the classical theory.

We shall call a family of close $h-$local charts on a small domain $U^*(\varepsilon)$ a \emph{local pseudo-chart} of the spectrum.
We notice that each $h-$local chart is normally valid for one good rectangle. However, an important fact we show in this paper that we may build a local pseudo-chart such that the leading term in asymptotic expansions of their $h-$local charts in the small parameters $h, \varepsilon $ is locally well defined on $U^*(\varepsilon)$.
The construction of local pseudo-charts will be done in detail in Section \ref{sp-lattice Sec}.

With regard to the global problem, we consider the spectrum as a discrete subset of the complex plane. We can apply a result of \cite{QS14} about a discrete set, called an \emph{asymptotic pseudo-lattice} (see also Definition \ref{pseu-lattice}), for the spectrum. It shows that the differential of the transition maps between two overlapping local pseudo-charts $ (f_i, U^i(\varepsilon)) $ and $ (f_ j, U^j (\varepsilon))$ is in the group $GL(2,\mathbb Z)$, modulo $\mathcal O(\varepsilon, \frac{h}{\varepsilon})$:
    $$ d (\widetilde{f}_i) = M_{ij} d( \widetilde{f}_j)+ \mathcal O(\varepsilon, \frac{h}{\varepsilon}),$$
where $\widetilde{f}_i= f_i \circ \chi$, $\widetilde{f}_j= f_j \circ \chi$, with $\chi$ is the function $ (u_1,u_2) \mapsto (u_1,\varepsilon u_2)$, and $M_{ij} \in GL(2, \mathbb Z)$ is an integer constant matrix.

Let $U (\varepsilon)$ be a bounded open domain in the spectral band and cover it by an arbitrary (small enough) locally finite covering of local pseudo-charts
$ \{ \left(  f_ j ,  U^j (\varepsilon) \right) \}_{j \in \mathcal{J} }$, here $\mathcal{J}$ is a finite index set. Then the spectral monodromy on $U (\varepsilon)$ is defined as the unique $1$-cocycle $\{  M_{ij} \} $, modulo coboundary in the first \v{C}ech cohomology group.
~\\
It is clear that if the spectral monodromy is not trivial, then the transition maps either, and the spectrum hasn't a smooth global lattice structure.


\section{Some basics}

\subsection{Weyl-quantization}  \label{sec2.1}
We will work throughout this article with pseudodifferential operators obtained by the $h-$Weyl-quantization of a standard space of symbols on $T^*M =\mathbb R^{2n}_{(x,\xi)}$, here $M= \mathbb R^n$ or a compact manifold of $n$ dimensions, and in particular $n=2$. We denote $\sigma $ the standard symplectic $2-$form on $T^*M$.

In the following we represent the quantization in the case $M= \mathbb R^n$. In the manifold case, the quantization is suitably introduced.
We refer to Refs. \cite{Dimas99}-\cite{Shubin01} for the theory of pseudodifferential operators.


        \begin{defi} \label{fonc ord}
                        A function $m: \mathbb R^{2n} \rightarrow (0, + \infty)$ is called an order function
                     if there are constants  $C,N >0$ such that
                                $$m(X)  \leq C \langle X-Y\rangle^{ N} m(Y), \forall X,Y \in \mathbb R^{2n},$$
        with notation $\langle Z\rangle= (1+ |Z|^2)^{1/2}$ for $Z \in \mathbb R^{2n}$.
        \end{defi}


         \begin{defi}
                        Let $m$ be an order function and $k \in \mathbb R$, we define classes of symbols of $h$-order $k$, denoted by $S^k(m)$ (families of functions), of $(a(\cdot;h))_{h \in (0,1]}$ on $\mathbb R^{2n}_{(x,\xi)}$ by
                        \begin{equation}
                                S^k(m)= \{ a \in C^\infty (\mathbb R^{2n})
                                 \mid  \forall \alpha \in \mathbb N ^{2n}, \quad |\partial^\alpha a | \leq  C_\alpha h^k m \} ,
                        \end{equation}
         for some constant $C _\alpha >0$, uniformly in $h \in (0,1]$. \\
         A symbol is called $\mathcal O(h^\infty)$ if it's in $\cap _{k \in \mathbb R } S^k(m):= S^{\infty}(m) $.
        \end{defi}

        Then $ \Psi^k(m)(M)$ denotes the set of all (in general unbounded) linear operators $A_h$ on $L^2(\mathbb R^n)$, obtained from the $h-$Weyl-quantization of symbols $a(\cdot;h) \in S^k(m) $ by the integral:
        \begin{equation} \label{symbole de W}
                            (A_h u)(x)=(Op^w_h (a) u)(x)= \frac{1}{(2 \pi h)^n}
                                 \int_{ \mathbb R^{2n}} e^{\frac{i}{h}(x-y)\xi}
                                 a(\frac{x+y}{2},\xi;h) u(y) dy d\xi.
        \end{equation}

In this work, we always assume that symbols admit classical asymptotic expansions in integer powers of $h$.
The leading term in this expansion is called the principal symbol of operators.

\subsection{Classical theory}

We recall here some notions and results from the classical theory.

\begin{defi}
 An integrable system on a symplectic manifold $(W, \sigma)$ of dimension $ 2n $ ($n \geq 1$) is given $n$ smooth real-valued functions $f_1, \dots, f_n$ in involution with respect to the Poisson bracket generated from the symplectic form $ \sigma$, whose differentials are almost everywhere linearly independent.
In this case, the map $$ F=(f_1, \dots, f_n ): M \rightarrow \mathbb R^n $$ is also called an \textit{integrable system}, or a \textit{momentum map}.

A smooth function $f_1$ is called completely integrable if there exists $n-1$ functions $f_2, \dots, f_n$ such that $F=(f_1, \dots, f_n )$ is an integrable system.
\end{defi}

Let $U$ be an open subset of regular values of $F$. Then we have,

\begin{theo}[Angle-action theorem] (Refs. \cite{Arnold67}, \cite{Arn63}, and \cite{Duis80})  \label{A-A}
Let $c \in U$, and $\Lambda_c$ be a compact regular leaf of the fiber $F^{-1}(c)$. Then there exists an open neighborhood $V= V^c$ of $\Lambda_c$ in $W$ such that $F\mid_ {V} $ defines a smooth locally trivial fibre bundle onto an open neighborhood $ U^c \subset U$ of $c$, whose fibres are invariant Lagrangian $n-$tori. Moreover, there exists a symplectic diffeomorphism $\kappa = \kappa^c$,
    $$\kappa= (x,\xi): V \rightarrow \mathbb T ^n \times  A, $$
with $A= A^c \subset \mathbb R^n$ is an open subset, such that $F\circ \kappa^{-1}(x, \xi)= \varphi(\xi)$ for all $x \in \mathbb T^n $,
and $\xi \in A$, and here $\varphi= \varphi^c: A \rightarrow \varphi (A)=U^c $ is a local diffeomorphism. We call $( x, \xi)$ local angle-action variables near $\Lambda_c$ and $(V, \kappa)$ a local angle-action chart.
\end{theo}

Notice that one chooses usually the local chart such that the torus $\Lambda_c$ is sent by $\kappa$ to the zero section $T ^n \times \{0 \}$. By this theorem, for every $a \in U^c $, then $\Lambda_a:= F^{-1}(a) \cap V^c$ is an invariant Lagrangian $n-$torus, called a Liouville torus, and we write
   \begin{equation} \label{tori} \Lambda_a \simeq \kappa(\Lambda_a)=\mathbb T ^n \times \{ \xi_a \}:= \Lambda_{\xi_a}, \end{equation}
with some $\xi_a \in A$.


\section{Spectral asymptotics}  \label{s a}
In this section, we shall apply the spectral asymptotic theory from \cite{Hitrik07} for small non-selfadjoint perturbations of selfadjoint operators in two dimensions to give the asymptotics of eigenvalues located in some suitable small windows of the complex plane.

    \subsection{General assumptions}  \label{hypothese}

 We give here some global geometric assumption on the dynamics of the principal symbol of unperturbed selfadjoint operators, as in Refs. \cite{Hitrik07}, and \cite{Hitrik04}.

    $ M $ denotes $\mathbb R^2$ or a connected compact analytic real (Riemannian) manifold of dimension $ 2 $ and we denote by $\widetilde{M}$
    the canonical complexification of $ M $, which is either $ \mathbb C ^ 2 $ in the Euclidean case or a Grauert tube in the case of manifold
    (see Ref. \cite{Burns01}).


    We consider a non-selfadjoint $ h-$pseudodifferential operator $P_{\varepsilon}$ on $ M $ and suppose that
    \begin{equation} \label{fm aa}
         P_{\varepsilon=0}:= P \quad \textrm{is formally selfadjoint}.
    \end{equation}

    Note that if $ M = \mathbb R ^ 2 $, the volume form $ \mu (dx) $ is naturally induced by the Lebesgue measure on $ \mathbb R ^ 2 $.
    If $ M $ is a compact Riemannian manifold, then the volume form $ \mu(dx) $ is induced by the given Riemannian structure of $ M $.
    Therefore in both cases the volume form is well defined and the operator $P_{\varepsilon}$ may be seen as an (unbounded) operator
    on $L^2(M, \mu(dx)) $.
    We always denote the principal symbol of $ P_ {\varepsilon} $ by $ p_ \varepsilon $ which is defined on $ T ^ * M $.

    We will assume the ellipticity condition at infinity for $ P_ {\varepsilon} $ at some energy level $E \in \mathbb R$ as follows:

    When $M=\mathbb R^2$, let
                        \begin{equation} \label{dkso2}
                            P_{\varepsilon}= P(x,hD_x,\varepsilon; h )
                        \end{equation}
   be the Weyl quantification of a total symbol $P(x, \xi,\varepsilon; h )$ depending smoothly on $\varepsilon$ in a neighborhood of $(0, \mathbb R) $ and taking values in the space of holomorphic functions of $(x,\xi)$ in a tubular neighborhood of $\mathbb R^4$ in $\mathbb C^4$ on which we assume that:
                \begin{equation}
                           | P(x, \xi,\varepsilon; h ) | \leq  \mathcal O(1) m(Re(x,\xi)).
                \end{equation}
    Here $ m $ is an order function in the sense of Def. \ref{fonc ord}.
    We assume moreover that $ m> 1 $ and $ P_ {\varepsilon} $ is classical of order $0$,
     \begin{equation}
        P(x, \xi,\varepsilon; h ) \sim \sum_{j=0}^\infty
                             p_{j,\varepsilon}(x,\xi) h^j, h \rightarrow 0,
     \end{equation}
     in the selected space of symbols as in Section \ref{sec2.1}.
     In this case, the principal symbol $p_\varepsilon = p_{0,\varepsilon}$ is the first term of the above expansion,
     and the ellipticity condition at infinity is
                        \begin{equation}    \label{el con}
                            |p_{\varepsilon}(x,\xi) - E| \geq \frac{1}{C} m(Re(x,\xi)), \mid (x,\xi)\mid \geq C,
                        \end{equation}
     for some $ C> 0 $ large enough.

     When $ M $ is a compact manifold, we consider $ P_ \varepsilon $ a differential operator on $ M $ such that in local coordinates $ x $ of $ M $, it is of the form:
        \begin{equation}
                            P_\varepsilon = \sum_{|\alpha |\leq m} a_{\alpha,\varepsilon}(x;h)(hD_x)^\alpha,
        \end{equation}
  where $D_x= \frac{1}{i} \frac{\partial}{\partial x}$ and $a_{\alpha,\varepsilon}$ are smooth functions of $\varepsilon$ in a neighborhood of $0$
   with values in the space of holomorphic functions on a complex neighborhood of $x=0$. \\
  We assume that these $a_{\alpha,\varepsilon}$ are classical of order $0$,
                         \begin{equation}
                           a_{\alpha,\varepsilon}(x;h) \sim \sum_{j=0}^\infty
                             a_{\alpha,\varepsilon,j}(x) h^j, h \rightarrow 0,
                        \end{equation}
  in the selected space of symbols.
  In this case, the principal symbol $p_\varepsilon$ in the local canonical coordinates associated $(x,\xi) $ on $T^*M $ is \begin{equation}
                           p_\varepsilon(x,\xi)= \sum_{|\alpha | \leq m} a_{\alpha,\varepsilon,0}(x) \xi^{\alpha},
  \end{equation}
  and the elipticity condition at infinity is
            \begin{equation} \label{el con2}
                            |p_{\varepsilon}(x,\xi) -E | \geq \frac{1}{C} \langle \xi \rangle ^m, (x,\xi) \in T^*M, \mid \xi \mid \geq C,
            \end{equation}
  for some $C>0$ large enough.
  Notice here that $ M $ has a Riemannian metric, then $\mid \xi \mid$ and $\langle \xi \rangle= (1+  \mid \xi \mid ^2)^{1/2}$ are well defined.

  It is known from Refs. \cite{Hitrik07}, and \cite{Hitrik04} that with the above conditions, the spectrum of $ P_ \varepsilon $ in
  a small but fixed neighborhood of $ E$ in $\mathbb C $ is discrete,
  when $h>0, \varepsilon \geq 0$ are small enough. Moreover, this spectrum is contained in a horizontal band of size $\varepsilon$:
    \begin{equation} \label{band}   |\mathrm{Im} (z)| \leq \mathcal O(\varepsilon ).\end{equation}

Let $p= p_{\varepsilon=0}$, it is principal symbol of the selfadjoint unperturbed operator $ P $
and therefore real. And let $q=\frac{1}{i}(\frac{\partial p_\varepsilon}{\partial \varepsilon})_{\varepsilon
                    =0}$ and assume that $q$ is a bounded analytic function on $T^*M$.
We can write the principal symbol
                            \begin{equation}  \label{symb prin}
                                  p_\varepsilon=p+i \varepsilon q+ \mathcal O (\varepsilon ^2).
                            \end{equation}

We assume that $ p $ is completely integrable, i.e., there exists an integrable system
\begin{equation} \label{F}
F=(p,f): T^*M \rightarrow \mathbb R^2
\end{equation}
Then the space of regular leaves of $F$ is foliated by Liouville Lagrangian invariant tori by Theorem \ref{A-A}.

We assume also that
  \begin{equation}p^{-1}(E) \cap T^*M  \ \textrm{is connected}, \end{equation}
and the energy level $ E $ is regular for $ p $, i.e., $dp \neq 0$ on $p^{-1}(E) \cap T^*M$. We would like to notice that the level set $p^{-1}(E)$ is compact, due to the ellipticity condition at infinity \eqref{el con} or \eqref{el con2}

Then the energy space $p^{-1}(E)$ is decomposed into
a singular foliation:
 \begin{equation} \label{dk18}
 p^{-1}(E) \cap T^*M   = \bigcup_{a \in J} \Lambda_a ,\end{equation}
 where $ J $ is assumed to be a compact interval, or, more generally, a connected graph with a finite number of vertices and of edges, see pp. 21-22 and 55 of Ref. \cite{Hitrik07}.

We denote by $S$ the set of vertices.
For each $a \in J$, $\Lambda_a$ is a connected compact subset invariant with respect to $H_p$.
Moreover, if $a \in J\backslash S$, $\Lambda_a$ are the Liouville tori depending
analytically on $a$. These tori are regular leaves corresponding to regular values of $F$.
Each edge of $ J $ can be identified with a bounded interval of $ \mathbb R $ and we have therefore
a distance on $J$ in the natural way.




We denote $H_p$ the Hamiltonian vector field of $p$, defined by $\sigma (H_p, \cdot)= -dp (\cdot)$.
For each $a \in J$, we define a compact interval in $\mathbb R$:
                            \begin{equation} \label{Q vo cung}
                                    Q_\infty(a)=
                                     \big [ \lim_{T\rightarrow \infty} \inf_{\Lambda_a} Re \langle q \rangle _T,
                                      \lim_{T\rightarrow \infty} \sup_{\Lambda_a} Re \langle q \rangle _T\big],
                            \end{equation}
where $\langle q \rangle _T$, for $T>0$, is the symmetric average time $T$ of $q$ along the
$H_p-$flow, defined by
                            \begin{equation}  \label{t-average}
                                 \langle q \rangle _T= \frac{1}{T} \int_{-T/2}^{T/2} q \circ exp(t H_p) dt.
                            \end{equation}
Then we can improve \eqref{band} that the spectrum the of $P_\varepsilon$ in the neighborhood of $E $ in $\mathbb C $ is located in the band
                      \begin{equation}  \label{loca. spectre 2}
                           \mathrm{Im} \left( \sigma(P_\varepsilon) \cap \{z \in \mathbb C: |Re z -E| \leq \delta \} \right) \subseteq
                             \varepsilon \big [ \inf \bigcup_{a \in J}Q_\infty(a)-o(1),
                              \sup \bigcup_{a \in J}Q_\infty(a) +  o(1) \big ],
                      \end{equation}
when $\varepsilon, h, \delta \rightarrow 0$ (see Ref. \cite{Hitrik07}).

From now, for simplicity, we will assume that $q$ is \textbf{real valued} (in the general case, simply replace $q$ by its real part $\mathrm{Re}(q)$ in regard to \eqref{symb prin}).

Each torus $\Lambda_a$, with $ a  \in J\backslash S$, locally can be embedded
in a Lagrangian foliation of $H_p-$invariant tori. By Theorem \ref{A-A}, there are
analytic local angle-action coordinates on an open neighborhood $V$ of $\Lambda_a$
\begin{equation}  \label{coor} \kappa= (x,\xi): V \rightarrow \mathbb T ^2 \times  A, \end{equation}
such that $\Lambda_a \simeq \Lambda_{\xi_a}$, $\xi_a \in A$, and $p$ becomes only a function of $\xi$,
           \begin{equation} \label{p}
                    p\circ \kappa^{-1} =p(\xi)= p(\xi_1, \xi_2), \ \xi \in A.           \end{equation}

Let $\Lambda \subset V $ be an arbitrary Liouville torus (close to $\Lambda_a$). We have
\begin{equation}  \label{Lam}
\Lambda \simeq \Lambda_{\xi}, \xi \in A .
\end{equation}
We define $\langle q \rangle $-the average of $ q $ on the torus $\Lambda$, with respect to the natural Liouville measure on $\Lambda$, denoted by $\langle q \rangle_{\Lambda} $, as following
                      \begin{equation} \label{moyenne de q}
                       \langle q \rangle_{\Lambda}= \int_{\Lambda}q .\end{equation}

\begin{rema}
 In the action-angle coordinates $(x,\xi)$ given by \eqref{coor}, we have
 \begin{equation} \label{moyenne2}
 \langle q \rangle_{\Lambda} =  \langle q \rangle_{\Lambda_ \xi}= \langle q \rangle (\xi)=
 \frac{1}{(2\pi)^2}\int_{\mathbb{T}^2}q(x,\xi)dx, \ \xi \in A.
 \end{equation}
In particular, $\langle q \rangle_{\Lambda_a}=\langle q \rangle(\xi_a)$.
\end{rema}

It is true that $\langle q \rangle_{\Lambda_a} $ depends analytically on $a \in J\backslash S$, and
we assume that it can be extended continuously on $J$. Furthermore, we assume that the function $a
\mapsto \langle q \rangle_{\Lambda_a} = \langle q \rangle(\xi_a)$ is not identically constant on
any connected component of $J\backslash S$, and that
\begin{equation}
\textrm{$dp(\xi)$ and $d \langle q \rangle (\xi)$ are $\mathbb R-$linearly independent at $\xi_a$.}
\end{equation}

Let $\Lambda \subset V $ be a Liouville torus as in \eqref{Lam}.
 Then we define the frequency of $\Lambda$ (also of $\Lambda_{\xi}$) by
\begin{equation} \label{frequence}
\omega(\xi)= \frac{\partial p}{\partial \xi} (\xi)= \big ( \frac{\partial p}{\partial \xi_1} (\xi),
 \frac{\partial p}{\partial \xi_2} (\xi) \big ), \ \xi \in A ,
 \end{equation}
 and the rotation number of $\Lambda$ by
 \begin{equation} \label{rho}
 \rho(\xi)=  \big [ \frac{\partial p}{\partial \xi_1} (\xi):
 \frac{\partial p}{\partial \xi_2} (\xi) \big ], \ \xi \in A ,
 \end{equation}
 viewed as an element of the real projective line. It is clear that $\rho$ depends analytically on $\xi \in A$ and we shall assume that the restricted function
 \begin{equation} \label{w}
\textrm{ $a \mapsto \rho(a):= \rho(\xi_a)$ is not identically constant on any connected
component of $J \backslash S$.}
\end{equation}

\begin{rema}[see pp. 56-57 of Ref. \cite{Hitrik07}]  \label{rem1}
For $a \in J\backslash S$, if $\rho(a) \notin \mathbb{Q}$, that means the frequency $\omega(\xi_a)$
is non resonant, then along the torus $\Lambda_a$, the Hamiltonian flow of $p$ is ergodic. Hence the
limit of $\langle q \rangle _T$, when $ T \rightarrow \infty$ exists, and is
equals to $ \langle q \rangle_{\Lambda_a}$. Therefore we
have $$Q_\infty(a)= \{\langle q \rangle_{\Lambda_a}\}.$$
\end{rema}


\subsection{Asymptotics of eigenvalues}

The spectral asymptotic theory (see Refs. \cite{Hitrik04}, and \cite{Hitrik07}) allows us to give the asymptotic description of all the eigenvalues
of $P_\varepsilon$ in some adapted small complex windows of the spectral band, which are associated with
Diophantine tori in the phase space. The force of the perturbation $\varepsilon $ is small and can be dependent or independent of the classical parameter $h$.
However in this work we consider the case when $\varepsilon$ is sufficiently small, dependent
on $h$, and in the following regime
\beq  \label{reg}
 h \ll \varepsilon = \mathcal{O}(h^\delta),
\eeq
 where $\delta >0$ is some number small enough but fixed.
In this case, spectral results are related to $(h,\varepsilon)$-dependent small windows.
\begin{defi}  \label{diop}
 Let $\alpha >0 $, $d>0$, and $\Lambda \simeq \Lambda_{\xi} $ be a $H_p-$invariant Lagrangian torus, as in \eqref{Lam}.
 We say that $\Lambda$ is $(\alpha,d)-$Diophantine if its frequency $\omega(\xi)$, defined in
 (\ref{frequence}), satisfies
 \begin{equation}  \label{dn alpha-d dioph}
     \omega(\xi) \in D_{\alpha,d}= \big \{   \omega \ \in \mathbb R^2 \big |  \   | \langle \omega,k \rangle |  \geq \frac{\alpha}{ |k|^{1+d}}, \
     \forall \ k \in \mathbb Z^2 \backslash \{ 0 \} \big \}.
 \end{equation}
\end{defi}
Notice also that when $d>0$ is fixed, the Diophantine property (for some $\alpha >0$) of $\Lambda$ is independent of selected angle-action coordinates, see \cite{Bost86}. If $\Lambda$ is $(\alpha,d)-$Diophantine, then its frequency must be irrational and we may have the result in Remark \ref{rem1}.

It is known that the set $D_{\alpha,d} $ is a closed set with closed half-line structure. When we take
$\alpha$ to be sufficiently small, it is a nowhere dense set but with no isolated points. Moreover, its
measure tends to large measure as $\alpha$ tends to $0$: the measure of its complement is of order
$\mathcal{O}(\alpha)$. See Refs. \cite{HB90}, and \cite{Poschel01}.
\begin{defi}  \label{dn bonnes valeurs}
        For some $\alpha>0 $ and some $d>0$, we define the set of good values associated with an energy level $E$, denoted by $\mathcal{G}(\alpha,d, E)$, obtained from $\cup_{a \in J}Q_\infty(a)$ by removing the following set of bad values $\mathcal{B}(\alpha,d, E)$:
\beas
\mathcal{B}(\alpha,d, E)& =  & \Bigg ( \bigcup_{dist(a,S) < \alpha} Q_\infty(a) \Bigg )
                                \bigcup \Bigg ( \bigcup_{a \in J\backslash S: \ \omega(\xi_a) \textrm{ is not }
                                    (\alpha,d)- \textrm{Diophantine}} Q_\infty(a)\Bigg )
                                          \\
                  &         &  \bigcup \Bigg( \bigcup_{a \in J\backslash S: \ |d \langle q \rangle(\xi_a) | < \alpha }
                                            Q_\infty(a) \Bigg )
 \bigcup \Bigg ( \bigcup_{a \in J\backslash S: \ |\omega'(\xi_a)| < \alpha } Q_\infty(a) \Bigg ) .
\eeas
\end{defi}

\begin{rema} \label{rem2}  We note that
         \begin{enumerate}[(i)] \label{peti mesure}
         \item When $d>0$ is kept fixed, the measure of the set of bad values
             $\mathcal{B}(\alpha,d, E)$ in $\cup_{a \in J}Q_\infty(a)$ (and in $\langle q
             \rangle_{\Lambda_a} (J) $) is $\mathcal O
             (\alpha)$, when $\alpha >0$ is small enough, provided that the measure of
                                    \begin{equation}  \label{a condition}
                                            \Bigg ( \bigcup_{a \in
                                             J\backslash S: \ \omega(\xi_a) \in \mathbb Q } Q_\infty(a)
                                             \Bigg ) \bigcup \Bigg ( \bigcup_{a \in S} Q_\infty(a)
                                            \Bigg )
                                    \end{equation}
             is sufficiently small, depending on $\alpha$ (see Ref. \cite{Hitrik07}).
             \item Let $G \in \mathcal{G}(\alpha,d, E) $ be a good value, then by Definition of
                 $\mathcal{B}(\alpha,d, E) $ and Remark \ref{rem1}, there are a finite number
                 of corresponding $(\alpha, d)-$Diophantine tori
                 $\Lambda_{a_1},\ldots,\Lambda_{a_L}$, with $L \in \mathbb N^*$ and $\{
                 a_1,\ldots,a_L\} \subset J\setminus S $, in the energy space $p^{-1}(E) \cap
                 T^*M$, such that the pre-image
                                         $$\langle q \rangle ^{-1}(G)= \{\Lambda_{a_1},\ldots, \Lambda_{a_L} \}.$$
             In this way, when $G $ varies in $\mathcal{G}(\alpha,d, E) $, we obtain a family of large measure
             of $(\alpha, d)-$Diophantine invariant tori  in the phase space
                 satisfying $ \{p= E, \langle q \rangle = G \}$.
         \end{enumerate}
\end{rema}
For $G \in \mathcal{G}(\alpha,d, E) $ is a good value, we define in the horizontal band of size
$\varepsilon$ of complex plan, given in \eqref{loca. spectre 2}, a suitable window of size
$\mathcal{O}(h^\delta) \times \mathcal{O}(\varepsilon h^\delta)$, around the \textit{good center $E+ i\varepsilon G$},
called a \textit{good rectangle},
 \begin{equation} \label{cua so}
                        R^{(E, G)}(\varepsilon,h)
                                    = (E+i \varepsilon \ G)+  \Big[-\frac{h^\delta}{\mathcal{O}(1)},\frac{h^\delta}{\mathcal{O}(1)} \Big]
                                    +i \varepsilon \Big [ -\frac{h^\delta}{\mathcal{O}(1)}, + \frac{h^\delta}{\mathcal{O}(1)} \Big ].
 \end{equation}
Now let $G \in  \mathcal{G}(\alpha,d, E)$ be a good value. As in Remark \ref{rem2}, there exists
$L$ elements in pre-image of $G$ by $\langle q \rangle$. We shall assume that $L=1$ and we write
    \begin{equation} \label{preimage}
     \langle q \rangle ^{-1}(G)=\Lambda_ a \subset  p^{-1}(E) \cap T^*M , \ a \in J\setminus S ,
    \end{equation}
Note that this hypothesis can be achieved if we assume that the function $(p, \langle q \rangle)$ is proper with connected fibres.


Let $\Lambda_a$, $ a  \in J\backslash S$ be an invariant Lagrangian torus and let $\kappa$ be
the action-angle local coordinates in \eqref{coor}. The fundamental cycles $(\gamma_1,\gamma_2)$ of
$\Lambda_a$ are defined by
            $$ \gamma_j= \kappa^{-1} ( \{ (x, \xi) \in T^* \mathbb{T}^2: x_j=0, \xi= \xi_a \}) , \ j=1,2. $$
Then we note $\eta \in \mathbb Z^2$ the Maslov indices and $S \in \mathbb R^2$ the action integrals of
these fundamental cycles,
\begin{equation} \label{act}
    S= (S_1, S_2) = \left( \int_{\gamma_1} \theta , \int_{\gamma_2} \theta  \right),
\end{equation}
where $\theta$ is the Liouville $1-$form on $T^*M$.

\begin{defi}[Refs. \cite{lectureColin}, and \cite{Rob93}]   \label{ind}
Let $ E $ be a symplectic space and let $ \Lambda (E) $ be his Lagrangian Grassmannian (which is
set of all Lagrangian subspaces of $ E $). We consider a bundle $ B $ in $ E $ over the circle or a
compact interval provided with a Lagrangian subbundle called vertical. Let $ \lambda (t) $ be a
section of $ \Lambda (B) $ which is transverse to the vertical edges of the interval in the case
where the base is an interval. The Maslov index of $ \lambda (t) $ is the intersection number of
this curve with the singular cycle of Lagrangians which do not cut transversely the vertical
subbundle.
\end{defi}
The Maslov index appears in the statement of spectral asymptotic results. In our work it is treated as a standard constant.


We recall here spectral asymptotic results for the standard case at the energy level
$E=0$, cited from \cite{Hitrik07}. However these results can be immediately generalized for any energy
lever $E$ by a translation.
\begin{theo}[\cite{Hitrik07}]  \label{theorem quasi-spectre}

    For $E=0$. We suppose that $P_\varepsilon$ is an operator satisfying
    Assumptions \ref{hypothese} (from \eqref{fm aa} to \eqref{w}), and in the regime $h \ll \varepsilon = \mathcal{O}(h^\delta)$ for $0< \delta <1 $.
  Let $G \in  \mathcal{G}(\alpha,d, 0)$ be a good value as Def. \ref{dn bonnes valeurs}, and assume that \eqref{preimage} is true. Suppose that the action-angle coordinates $\kappa$ in \eqref{coor} send $\Lambda_a$ to the zero section $ \mathbb T^2 \times \{\xi_a=0\} \in T^* \mathbb T^2$.
    Assume that $dp(\xi)$ and $d \langle q \rangle (\xi)$ are $\mathbb R-$linearly independent at $\xi_a=0$, where $p=p(\xi)$ as in \eqref{p} and $\langle q \rangle (\xi) $  is the average of $q$ on tori, given in (\ref{moyenne2}).

  Then the eigenvalues $\mu$ of $ P_ \varepsilon $
  with multiplicity in the good rectangle $  R^{(0, G)}(\varepsilon,h)$ of the form \eqref{cua so} have the following
  expression,
  \begin{equation} \label{eigenvalues}
  \mu= P ^{(\infty)} \Big( h(k-\frac{\eta}{4})-\frac{S}{2 \pi}, \varepsilon; h\Big) + \mathcal O(h^\infty),
                                        k \in \mathbb Z^2,
  \end{equation}
    where $P^{(\infty)}(\xi, \varepsilon; h )$ is a smooth function of $\xi$ defined in a neighborhood of $(0, \mathbb R^2)$
    and $\varepsilon, h$ in neighborhoods of $(0, \mathbb R)$.
    Moreover $P^{(\infty)}(\xi, \varepsilon; h )$ is real valued for $\varepsilon =0$, admits the following polynomial asymptotic
    expansion in $(\xi, \varepsilon, h)$ for the $C^\infty-$ topology:
    \begin{equation} \label{symbole normal}
        P^{(\infty)} ( \xi, \varepsilon; h) \sim    \sum_{\alpha, j,k} C_{\alpha j k} \ \xi^\alpha \varepsilon^ j h^k.
     \end{equation}
   In particular $P^{(\infty)}$ is classical in the space of symbols with $h-$leading term:
    \begin{equation} \label{prin normal}
            p_{0}^{(\infty)} (\xi, \varepsilon)= p(\xi)+ i \varepsilon \langle q \rangle (\xi) +
                                             \mathcal O(\varepsilon ^2).
    \end{equation}

\end{theo}

The fact that the horizontal spectral band is of size $\mathcal{O}(\varepsilon)$ suggests us
introducing the function
        \begin{eqnarray}   \label{chi}
            \chi :  \mathbb R^2 \ni u= (u_1,u_2) \mapsto \chi_u &=& (u_1, \varepsilon u_2) \in \mathbb R^2
            \\
            &\cong & u_1+i \varepsilon u_2 \in \mathbb C,     \nonumber
        \end{eqnarray}
in which we identify $\mathbb C$ with $\mathbb R^2$.

Noticing that $ \sigma(P_\varepsilon)= \sigma(P_\varepsilon - \chi_a  ) +\chi_a$ and applying Theorem \ref{theorem quasi-spectre}
(in the standard case) for the operator $(P_\varepsilon - \chi_a)$ with respect to the good rectangle $R^{(0, G)}(\varepsilon,h)$, we obtain easily the asymptotic eigenvalues of $P_\varepsilon$ in any good rectangle $R^{(E, G)}(\varepsilon,h)$, as following.
\begin{theo} \label{general quasi-spectre}
    For $E \in \R$. We suppose that $P_\varepsilon$ is an operator satisfying Assumptions \ref{hypothese} (from \eqref{fm aa} to \eqref{w}), and in the regime $h \ll \varepsilon = \mathcal{O}(h^\delta)$ for $0< \delta <1 $.
  Let $G \in  \mathcal{G}(\alpha,d, E)$ be a good value as Def. \ref{dn bonnes valeurs}, and assume that \eqref{preimage} is true. Suppose that by $\kappa$, given in \eqref{coor}, $\Lambda_a \simeq \Lambda_{\xi_a}$, with $\xi_a \in A$.
    Assume that $dp(\xi)$ and $d \langle q \rangle (\xi)$ are $\mathbb R-$linearly independent at $\xi_a$, where $p=p(\xi)$ as in \eqref{p} and $\langle q \rangle (\xi) $ as in (\ref{moyenne2}).

  Then the eigenvalues $\mu$ of $ P_ \varepsilon $
  with multiplicity in the good rectangle $  R^{(E, G)}(\varepsilon,h)$ of the form \eqref{cua so}, modulo $\mathcal O(h^\infty)$, are $h-$\textit{locally} given by
  \begin{equation} \label{eigenvalues-E}
  \mu= P ^{(\infty)} \Big( \xi_a+ h(k-\frac{\eta}{4})-\frac{S}{2 \pi},\varepsilon; h\Big) + \mathcal O(h^\infty),
                                        k \in \mathbb Z^2,
  \end{equation}
    where $P^{(\infty)}(\xi; \varepsilon, h )$ is a smooth function defined in a neighborhood of $(\xi_a, \mathbb R^2)$
    and of $\varepsilon, h$ in neighborhoods of $(0, \mathbb R)$.
    Moreover $P^{(\infty)}(\xi; \varepsilon, h )$ admits asymptotic expansions of the form \eqref{symbole normal}, with $h-$leading term of the form \eqref{prin normal}.
\end{theo}

\begin{rema} \label{tsymb}
We can write the total symbol $ P^{(\infty)}$ in the reduce form:
     \begin{equation}
         P^{(\infty)} ( \xi, \varepsilon; h) =  p(\xi)+ i \varepsilon \langle q \rangle (\xi) +   \mathcal O(\varepsilon ^2)+ \mathcal O(h),
    \end{equation}
uniformly for $\varepsilon$ and $h$ small.
\end{rema}

\begin{rema} \label{diffeo}
We can show that the function
$P^{(\infty)}$ in the above theorem is a local diffeomorphism from a neighborhood of $(\xi_a, \mathbb R^2)$ into its image, that is in a $\mathcal O(\varepsilon)-$ horizontal band (see Ref. \cite{QS14}).
Therefore the eigenvalues of $P_\varepsilon$ in a good rectangle form a deformed $h-$\textit{lattice}. It's the image of a square lattice of $h \mathbb Z^2$ by a local diffeomorphism.
Moreover, we can show that the lattice has a horizontal spacing $ \mathcal O(h)$
and a vertical spacing $ \mathcal O(\varepsilon h)$.
\end{rema}




\section{Spectral asymptotic pseudo-lattice and its monodromy}   \label{sp-lattice Sec}

In the previous section we have seen that the spectrum locally is a deformed $h-$ lattice. In this section we are studying globally the spectrum in showing that the spectrum is an asymptotic pseudo-lattice. Then we can define a monodromy for the spectrum.

\subsection{Monodromy of an asymptotic pseudo-lattice} \label{mapl}
We recall here the definition of an asymptotic pseudo-lattice and its monodromy given in \cite{QS14}. This is a discrete subset of $\mathbb R^2$ admitting a particular property.

For any subset $U$ of $\mathbb R^2$ we denote
\beq \label{elp}
 \ U(\varepsilon)=  \chi(U),
\eeq
where $\chi$ is the map in \eqref{chi}.
\begin{defi} \label{pseu-lattice}
     Let $U$ be an open subset of $\mathbb R^2$ with compact closure and
     let $\Sigma(\varepsilon, h)$ (which depends on small $h$ and $\varepsilon$) be a discrete subset of $U(\varepsilon)$.
     For $h, \varepsilon$ small enough and in the regime $ h \ll \varepsilon$, we say that $(\Sigma(\varepsilon, h), U(\varepsilon))$
     is an asymptotic pseudo-lattice if:
         for any small parameter $\alpha >0$, there exists a set of good values in $\mathbb R^2$, denoted by $\mathcal{G}(\alpha)$, whose complement is of size $\mathcal O(\alpha)$ in the sense:
             $$\mid {}^C \mathcal{G} (\alpha) \cap I  \mid \leq C \alpha \mid I \mid, $$
with a constant $C>0$ for any domain $I \subset \mathbb R^2$; for all $c \in U$, there exists a small open subset $U^c \subset U $ around $c$ such that for every good value $a \in U^c  \cap \mathcal{G}(\alpha)$, there is an adapted good rectangle $R^{(a)}(\varepsilon,h) \subset U^c( \varepsilon)$ of the form \eqref{cua so} (with $a=(E,G)$), and a smooth local diffeomorphism $f= f(\cdot, \varepsilon;h)$ which sends $R^{(a)}(\varepsilon,h)$ on its image, satisfying
\beq  \label{semi-cart}
    \Sigma( \varepsilon, h) \cap R^{(a)}(\varepsilon,h) \ni  \mu \mapsto  f(\mu, \varepsilon; h) \in h \mathbb Z^2 +\mathcal O(h^\infty).                        \eeq
Moreover, the function $\widetilde{f}:= f \circ \chi$, with $\chi$ defined by \eqref{chi}, admits an asymptotic expansion in
$(\varepsilon,\frac{h}{\varepsilon})$ for the $C ^\infty-$topology in a neighborhood of $a$,
uniformly with respect to the parameters $h$ and $\varepsilon$, such that its leading term $\widetilde{f}_0$
is a diffeomorphism, independent of $\alpha$, locally defined on whole $U^c $ and independent
of the selected good values $a \in U^c $.

We also say that the couple $(f(\cdot, \varepsilon;h), R^{(a)}(\varepsilon,h) )$ is a $h-$chart of $\Sigma( \varepsilon, h)$,
and the family of $h-$charts $(f(\cdot, \varepsilon;h), R^{(a)}(\varepsilon,h) )$, with all $a
\in U^c \cap \mathcal{G}(\alpha) $, is a local pseudo-chart on $U^c( \varepsilon)$ of
$(\Sigma(\varepsilon, h), U(\varepsilon))$.
\end{defi}

\begin{rema}
A standard lattice is obviously an asymptotic pseudo-lattice.
An another example is a known lattice with some similar but lighter properties, named asymptotic lattice, given in \cite{Vu-Ngoc99}.
That lattice is modeled on the joint spectrum of system of commuting operators. It is locally defined, while the asymptotic pseudo-lattice is very delicate, it is $h-$locally defined.

The introduction of this discrete lattice aims to show that the combinatorial invariant that we will
define is directly built from the spectrum of operators. If different operators have the same
spectrum, then they have the same monodromy.
\end{rema}


Let $\{ U^j \}_{j \in \mathcal{J} }$, here $\mathcal{J}$ is a finite index set, be an arbitrary (small enough) locally finite covering of $U$. Then the asymptotic pseudo-lattice $(\Sigma(\varepsilon, h), U(\varepsilon))$ can be covered by associated local pseudo-charts $ \{ \left(  f_ j (\cdot, \varepsilon; h),  U^j (\varepsilon) \right) \}_{j \in \mathcal{J} }$. Note from Definition \ref{pseu-lattice} that the leading terms $ \widetilde{f}_{j,0}(\cdot, \varepsilon; h)$ are well defined on whole $ U^j$ and we can see them as
the charts of $U$. Analyzing transition maps, we had the following result.
\begin{prop} [\cite{QS14}]   \label{dl}
 On each nonempty intersection $  U^i \cap U^j \neq \emptyset$, $i, j \in \mathcal{J}$, there exists an unique integer linear map $M_{ij} \in GL(2,
\mathbb Z)$ (independent of $h, \varepsilon$) such that:
\begin{equation} \label{transition-pseu}
d \big ( \widetilde{f}_{i, 0} \circ (\widetilde{f}_{j, 0})^{-1} \big )= M_{ij}.
\end{equation}
\end{prop}

Then we define the (linear) monodromy of the asymptotic pseudo-lattice $(\Sigma(\varepsilon, h), U(\varepsilon))$ as the $1$-cocycle of $\{M_{ij} \} $, modulo coboundary, in the \v{C}ech cohomology of $U$ with values in the integer linear group $GL(2, \mathbb Z)$,
denoted by
         \begin{equation} \label{mono}
             [\mathcal M] \in \check{H}^1(U,GL(2, \mathbb Z) ).
         \end{equation}
It does n't depend on the selected finite covering $\{ U^j \}_{j \in \mathcal{J} }$ and the small parameters $h, \varepsilon$.

We can also associate the class $[\mathcal M]$ with its holonomy, that is a group morphism from the fundamental group $\pi_1(U)$ to the group $ GL(2, \mathbb Z)$, modulo conjugation. The triviality of $[\mathcal M]$ is equivalent to the one of its holonomy.



\subsection{Spectral monodromy of $P_\varepsilon$}
Now let $F=(p,f)$ be an integrable system as in \eqref{F} and denote $U_r$ the set of all regular values of $F$.
We recall that the space of regular leaves of $F$ is foliated by Liouville invariant tori by Theorem \ref{A-A}.

\begin{theo} \label{mtheo1}
     Let $U$ be an open subset of $U_r$ with compact closure.
     In the regime $h \ll \varepsilon = \mathcal{O}(h^\delta)$ for $0< \delta <1 $, we suppose that $P_\varepsilon$ is an operator satisfying the assumptions of Sec. \ref{hypothese} (from \eqref{fm aa} to \eqref{w}) for any energy level $E$ in the projection of $U$ on horizontal axis.
     For any Liouville torus $\Lambda \subset F^{-1} (U) $ such that $\Lambda \simeq \Lambda_{\xi}$ in a local angle-action coordinate $\kappa=(x,\xi)$ as in \eqref{coor}, we assume that $dp(\xi)$ and $d \langle q \rangle (\xi)$ are $\mathbb R-$linearly independent, where $p=p(\xi)$ as in \eqref{p} and $\langle q \rangle (\xi) $ as in (\ref{moyenne2}) are expressions of $p$ and $\langle q \rangle$ in the local angle-action variables. Moreover we assume that the map $(p, \langle q \rangle )$ on $F^{-1} (U) $ is proper with connected fibers.
     Then $(\sigma(P_\varepsilon), U(\varepsilon) )$ is an asymptotic pseudo-lattice.
\end{theo}

\begin{proof}[Proof of Theorem \ref{mtheo1}]

Let an any point $c \in U$. There exists uniquely a Liouville torus $\Lambda_c= (p, \langle q \rangle )^{-1}(c)$ in $T^*M$.
We note that $p$ and $\langle q \rangle$ commute in neighborhood of each Liouville torus, due to the fact that $\langle q \rangle $ is invariant under the flow of $p$.
Therefore we can use a local action-angle coordinate in a neighborhood $V^c \subset T^*M$ of $\Lambda_c$ as in Theorem \ref{A-A} with $F=(p, \langle q \rangle)$ and $\varphi(\xi)= (p(\xi), \langle q \rangle (\xi)), \ \xi \in A^c$.

Now, for any point $a=(E,G) $ in a small enough neighborhood $U^c $ such that $G$ is a good value, i.e., $G \in \mathcal{G}(\alpha,d,
E)$, see Definition \ref{dn bonnes valeurs}. We notice also that with assumptions of the theorem, the condition \eqref{preimage} is valid ($L=1$). So the corresponding Liouville torus $\Lambda_a= F^{-1}(a)$ is $(\alpha,d)-$Diophantine, as discussed in Remark \ref{rem2}.
Suppose that $\Lambda_a \simeq \Lambda_{\xi_a}$, with $\xi_a= \varphi^{-1}(a) \in A^c$.
Then with the help of Theorem \ref{general quasi-spectre} the eigenvalues $\mu$ of $ P_ \varepsilon $ in the good rectangle $  R^{(E, G)}(\varepsilon,h)$ of the form \eqref{cua so}, are $h-$\textit{locally} given by \eqref{eigenvalues-E}.

On the other hand, we have a classical result that for any $a \in U^c$,
 \begin{equation}  \label{pt action}
                             \frac{S}{2 \pi}- \xi_a := \tau_c \in \mathbb R^2,
 \end{equation}
 is locally constant in $c \in U_r$ (see Ref. \cite{QS14}).

Then from Remark \ref{diffeo}, the Eq. \eqref{eigenvalues-E} provides a smooth local diffeomorphism at $E+i \varepsilon G \in \mathbb C$,
denoted by $f= f(\mu, \varepsilon; h)$, that sends $\mu$ to $hk \in h \mathbb Z^2 $, modulo $\mathcal O(h^\infty)$, of the form
 \begin{eqnarray}  \label{new hk}
                 f= f(\mu, \varepsilon;h) & =&  \frac{S}{2 \pi}- \xi_a + h \frac{\eta}{4}+ P^{-1}(\mu)
                  \\
                      &= & \tau_c + h \frac{\eta}{4}+ P^{-1}(\mu)
                       \nonumber \\
                \sigma(P_\varepsilon) \cap R^{(E, G)}(\varepsilon,h) \ni  \mu & \mapsto &
                        f(\mu,\varepsilon; h) \in h \mathbb Z^2 +\mathcal O(h^\infty),  \nonumber
 \end{eqnarray}
 here $P$ is the map $P ^{(\infty)}$ in Theorem \ref{eigenvalues-E}, for ease of notation.

We shall show that this map should be a $h-$chart of $\sigma(P_\varepsilon)$ on the good rectangle $R^{(E, G)}(\varepsilon,h)$.

Let $\widetilde{f}= f \circ \chi$, then we have
\begin{equation}    \label{f tilde}
\widetilde{f}=  \tau_c + h \frac{\eta}{4}+ P^{-1} \circ \chi .
\end{equation}
To analysis $P^{-1} \circ \chi$, we first consider its inverse, $\widehat{P}:= \chi^{-1} \circ
P$. It is obtained from $P$ by dividing the imaginary part of $P$ by $\varepsilon$. As $P$ admits
an asymptotic expansion in $(\xi, \varepsilon,h)$, so it is true that $\widehat{P}$ admits an
asymptotic expansion in $(\xi, \varepsilon,\frac{h}{\varepsilon})$ (here $h \ll \varepsilon $).
Moreover, from Remark \ref{tsymb} we can write $\widehat{P}$ in the reduce form:
\begin{eqnarray}
\widehat{P}(\xi, \varepsilon;   h) & = &  \widehat{P}_0(\xi)+ \mathcal O(\varepsilon) + \mathcal O(\frac{h}{\varepsilon})
                       \nonumber \\
                            & = & \widehat{P}_0(\xi) + \mathcal O(\varepsilon, \frac{h}{\varepsilon}),
\end{eqnarray}
uniformly for $h, \varepsilon $ small and $h \ll \varepsilon $, with $\widehat{P}_0(\xi)=
p(\xi)+ i \langle q \rangle (\xi) \cong \varphi(\xi)$.

The Proposition \ref{D.A inverse} (as below) ensures that the map $\widehat{P}^{-1}= P^{-1} \circ \chi$ admits an asymptotic expansion in $(\varepsilon,\frac{h}{\varepsilon})$ whose first term is $(\widehat{P}_0)^{-1}=  \varphi ^{-1}$.
Hence the map $\widetilde{f}$ in \eqref{f tilde} admits an asymptotic expansion in
$(\varepsilon, \frac{h}{\varepsilon})$ with the leading term
            \begin{equation}  \label{pre terme f tilde}
                 \widetilde{f}_0= \tau_c+  \varphi^{-1}.
            \end{equation}
It is clear that the leading term $\widetilde{f}_0$ is a local diffeomorphism, defined on
$U^c$. It does not depend on the selected good rectangle in $U^c ( \varepsilon )$.
So $(f,R^{(E, G)}(\varepsilon,h))$ is a $h-$chart of $\sigma(P_\varepsilon)$ on $U^c(\varepsilon )$.

On the other hand, in the domain $U^c(\varepsilon )$, there are many good rectangles, due to the fact that the set of good values is of large measure, see Remark \ref{rem2}. A family of such as these $h-$charts with common leading term satisfying \eqref{pre terme f tilde} forms a local pseudo-chart for $\sigma(P_\varepsilon)$ on $U^c(\varepsilon )$.

The above construction ensures that $(\sigma(P_\varepsilon), U(\varepsilon) )$ is an asymptotic pseudo-lattice.
\end{proof}


With the help of Theorem \ref{mtheo1} and the construction in Sec. \ref{mapl} we may now give the definition of the monodromy for operators as following.
\begin{defi} \label{monodromie spectrale}
For $\varepsilon, h > 0$ sufficiently small such that $h \ll \varepsilon = h^\delta, 0< \delta< 1$, we suppose that the domain $U$ and the operator $P_\varepsilon$ satisfies all assumptions of Theorem \ref{mtheo1}. Then we define the spectral monodromy of $P_\varepsilon$ on the domain $U(\varepsilon)$, denoted by $[\mathcal M _{sp}(P_\varepsilon)]$, as the monodromy of the asymptotic pseudo-lattice $(\sigma(P_\varepsilon), U(\varepsilon) )$, given in \eqref{mono}.
\end{defi}


\begin{prop}(see \cite{QS14})  \label{D.A inverse}
            Let $\widehat{P}= \widehat{P}(\xi; X)$ be a complex-valued smooth function of $\xi$
            near $0 \in \mathbb R^2$ and $X$ near $0 \in \mathbb R^n$.
            Assume that $\widehat{P}$ admits an asymptotic expansion in $X$ near $0$ of the form
                $$\widehat{P}(\xi; X) \sim  \sum_\alpha C_\alpha(\xi) X^\alpha,$$
            with $C_\alpha(\xi)$ are smooth functions and $C_0(\xi):=\widehat{P}_0(\xi) $ is a local diffeomorphism near $\xi=0$.

            Then, for $\mid X \mid $ small enough, $\widehat{P}$ is also a smooth local diffeomorphism near $\xi=0$ and its inverse admits an asymptotic expansion in $X$ near $0 $ whose the first term is
            $(\widehat{P}_0)^{-1}$.
\end{prop}


\subsection{Spectral monodromy recovers the classical monodromy}

Suppose that $\{ U^j \}_{j \in \mathcal{J} }$ is an arbitrary small enough finite open covering of $U$. Here $\mathcal{J}$ is a finite index set. With notations as in Theorem \ref{A-A}, $F^{-1}(U)$ can be covered by a finite covering of angle-action charts $ \{
(V^j, \kappa ^j ) \} _ {j \in \mathcal{J} }$.
Then the classical monodromy is defined as the $\mathbb Z^n-$ bundle $H_1(\Lambda_c, \mathbb Z) \rightarrow c \in U $, whose transition maps between trivializations are
$$\{ {}^t \big ( d( ({\varphi ^i}) ^{-1} \circ \varphi ^j  ) \big ) ^{-1} \}=
 \{ {}^t \big ( M_{ij}^{cl} \big ) ^{-1} \}.$$
Here $H_1(\Lambda_c, \mathbb Z)$ is the first homology group of the Liouville torus $\Lambda_c$, and
$M _{ij}^{cl} $ are a locally constant matrixes of $ GL(2, \mathbb Z)$ defined on the nonempty overlaps $V^i \cap V^j$, $i, j \in
\mathcal{J}$, see \cite{Duis80}.

Hence, in regard to the definition of the spectral monodromy by \eqref{mono}, \eqref{transition-pseu} and the result \eqref{pre terme f tilde}, we can state that the spectral monodromy of non-selfadjoint small perturbed operators allows to recover the classical monodromy of the underlying completely integrable system of selfadjoint unperturbed operators. More precisely we have.
\begin{theo}
The spectral monodromy of $P_\varepsilon$ is the adjoint of the classical monodromy defined by $p$.
\end{theo}
This result is seminar to an one of paper \cite{QS14}, operators in both cases are related to completely integrable systems.


In conclusion, the spectral monodromy of small non-selfadjoint perturbations of a selfadjoint operator, admitting a principal symbol that is completely integrable, is well defined directly from its spectrum, independent of small perturbations and the classical parameter.

Moreover, the spectral monodromy allows to recover the classical monodromy of integrable systems. It shows the important relationship between quantum and classical mechanics.


\section*{Acknowledgement} This work was completed during my visit to the Vietnam Institute for Advanced Study in Mathematics (VIASM). I would like to thank the Institute for its support and hospitality. I would like also to thank the Vietnam National University of Agriculture, who gave me a good opportunity to develop my research.


\bigskip

\end{document}